\documentclass[journal]{IEEEtran}
\usepackage{mathrsfs}
\usepackage{graphicx}
\usepackage{epstopdf}

\usepackage{subeqnarray}
\usepackage{cases}

\usepackage{slashbox}
\usepackage{amssymb} 
\usepackage{amsthm}
\usepackage[numbers,sort&compress]{natbib}
\newtheorem{theorem}{Theorem}
\newtheorem{remark}{Remark}
\newtheorem{example}{Example}
\newtheorem{lemma}{Lemma}
\newtheorem{definition}{Definition}
\newtheorem{corollary}{Corollary}

\ifCLASSINFOpdf

\else
 \fi
\begin{document}
%
\title{A novel perspective to gradient method:\\
 the fractional order approach}

\author{Yuquan~Chen,~
        Yiheng~Wei,~
		and~Yong~Wang,~\IEEEmembership{Senior Member,~IEEE,}
\thanks{*The work described in this paper was fully supported by the National Natural Science Foundation of China (No. 61573332) and the Fundamental Research Funds for the Central Universities (No. WK2100100028).}
\thanks{All the authors are with the Department of Automation, University of Science and Technology of China, Hefei 230027, China (E-mail: cyq@mail.ustc.edu.cn; neudawei@ustc.edu.cn; yongwang@ustc.edu.cn).}
}


\maketitle

\begin{abstract}
In this paper, we give some new thoughts about the classical gradient method (GM) and recall the proposed fractional order gradient method (FOGM). It is proven that the proposed FOGM holds a super convergence capacity and a faster convergence rate around the extreme point than the conventional GM. The property of asymptotic convergence of conventional GM and FOGM is also discussed. To achieve both a super convergence capability and an even faster convergence rate, a novel switching FOGM is proposed. Moreover, we extend the obtained conclusion to a more general case by introducing the concept of $p$-order Lipschitz continuous gradient and $p$-order strong convex. Numerous simulation examples are provided to validate the effectiveness of proposed methods.
\end{abstract}

\IEEEpeerreviewmaketitle

\section{Introduction}
 Fractional order calculus is a natural generalization of classical integer order calculus, which has developed for about three hundred years. Yet, it only developed as a pure mathematics due to the lack of physical meaning. Recently, it has brought a new avenue to the development to all kinds of fields, such as automatic control and system modeling. With the rapid growth of data, it is emergent to find an efficient method for signal processing and optimization. It is sure that fractional order calculus also brings new perspectives in developing new optimization algorithms that we mainly concern in this paper.

As a standard optimization algorithm, GM has been widely used in many engineering applications like adaptive filter \citep{vaudrey2003stability,lin2008new}, image processing \cite{kretschmer1978improved,glover1979high,balla2013fast}, system identification \citep{wong1999hybrid,angulo2015nonlinear,lin2015parameter}, iterative learning, and computation intelligence. The linear convergence rate is rigorously proven under the assumption that the function is strong convex in \cite{boyd2004convex}. However, the convergence rate around the extreme point is quite slow, which is undesired. To overcome the slow convergence rate, Newton algorithm is proposed, which modified the iterative direction at each step by multiplying the inverse Hessian matrix. Yet, Newton algorithm only suits for a strong convex function and the computation cost is quite huge, which restrict its usage a lot. Besides, the choice of step size is also a big problem in GM. In \cite{boyd2004convex}, the choice of step size is discussed under the strong convex condition, which gives the range of step size.

As pointed before, fractional order calculus may bring a new chance for GM such as improving the convergence rate around the extreme point and behaving a robust convergence capacity to step size. Yet, research of FOGMs is still in its infancy and deserves further investigation. In \cite{tan2015novel}, the authors proposed an FOGM by using Caputo's fractional order derivative with a gradient order no more than $1$ as the iteration direction, instead of an integer order derivative. It was found that a smaller weight noise can be achieved if a smaller gradient order is used, and the algorithm converges faster if a bigger gradient order is used. A similar idea can be found in \cite{pu2015fractional} where a different Riemann Liouville's fractional order derivative was used to develop a fractional steepest descent method. However, the algorithm in \cite{pu2015fractional} cannot guarantee the convergence to the exact extreme point. This shortcoming has been well overcome in \cite{chaudhary2015identification}. Despite some minor errors in using the Leibniz rule, the method developed in \cite{chaudhary2015identification} has been successfully applied in speech enhancement \cite{geravanchizadeh2014speech} and noise suppression \cite{shah2014fractional}.


It is worth pointing out that most existing work on FOGM in literature concerns the problem of quadratic function optimization only and may not even guarantee the convergence to the extreme point. Thus we have proposed a novel FOGM for a general convex function, which can guarantee the convergence capability. Yet, there are no detailed analysis about the proposed FOGM. Thus in this study, we carefully analyze the properties of the proposed FOGM, including convergence capability, convergence accuracy, and convergence rate. Based on the obtained properties, a novel switching FOGM is further proposed, which shows a great convergence capability and a faster convergence rate. Moreover, with the concepts of $p$-order Lipschitz continuous gradient continuous and $p$-order strong convex being defined, all the conclusion are extended to a more general case. Besides all the obtained meaningful results which promote the development of FOGM, the results also give us some new thoughts about the conventional GM, of which the most important is that why strong convexity is always needed for analyzing.

The remainder of the article is organized as follows. Section \ref{sec2} gives some basic definitions about fractional order calculus and convex optimization. Some introduction of existing FOGMs are also presented in Section \ref{sec2}. Properties of proposed FOGM are discussed in Section \ref{sec4}. A novel switching FOGM is presented in Section \ref{sec8}. In Section \ref{sec9}, the conclusion are extended to a more general case. Some simulation examples are provided to demonstrate the effectiveness of the proposed methods in Section \ref{sec5}. The article is finally concluded in Section \ref{sec6}.
\section{Preliminaries and proposed FOGM}\label{sec2}
Recall the definition of Lipschitz continuous gradient.
\begin{definition}\cite{boyd2004convex}
For a scalar function $f(t)$ whose first order derivative is existing, there exists a scalar $\mu>0$ such that
\begin{eqnarray}
\left| {{f^{\left( 1 \right)}}\left( x \right) - {f^{\left( 1 \right)}}\left( y \right)} \right| \le \mu \left| {x - y} \right|,
\end{eqnarray}
for any $x$ and $y$ belonging to the definition domain of $f(t)$. Then $f(t)$ is said to satisfy Lipschitz continuous gradient.
\end{definition}
The definition of fractional order Lipschitz continuous gradient is given as follows.
\begin{definition}
For a scalar function $f(t)$ whose first order derivative is existing, there exists a scalar $\mu>0$ such that
\begin{eqnarray}
\left| {{f^{\left( 1 \right)}}\left( x \right) - {f^{\left( 1 \right)}}\left( y \right)} \right| \le \mu \left| {x - y} \right|^p,
\end{eqnarray}
for any $x$ and $y$ belonging to some region of the definition domain of $f(t)$. Then $f(t)$ is said to satisfy local $p$-order Lipschitz continuous gradient.
\end{definition}
\begin{definition}\cite{boyd2004convex}
For a scalar convex function $f(t)$ whose first order derivative is existing, there exists a scalar $\lambda>0$ such that
\begin{eqnarray}
\left| {{f^{\left( 1 \right)}}\left( x \right) - {f^{\left( 1 \right)}}\left( y \right)} \right| \ge \lambda \left| {x - y} \right|,
\end{eqnarray}
for any $x$ and $y$ belonging to the definition domain of $f(t)$. Then $f(t)$ is said to be strong convex.
\end{definition}
The definition of fractional order strong convexity is given as follows.
\begin{definition}
For a scalar convex function $f(t)$ whose first order derivative is existing, there exists a scalar $\lambda>0$ and $p>0$ such that
\begin{eqnarray}
\left| {{f^{\left( 1 \right)}}\left( x \right) - {f^{\left( 1 \right)}}\left( y \right)} \right| \ge \lambda \left| {x - y} \right|^p,
\end{eqnarray}
for any $x$ and $y$ belonging to the definition domain of $f(t)$. Then $f(t)$ is said to be $p$-order strong convex.
\end{definition}

For any constant $n-1<\alpha<n,~n\in \mathbb{N^+}$, the Caputo's derivative \cite{podlubny1998fractional} with order $\alpha$ for a smooth function $f(t)$ is given by
\begin{eqnarray}\label{Caputo}
{}^{\rm{C}}_{t_0}{\mathscr D}_t^\alpha f\left( t \right) = \frac{1}{{\Gamma \left( {n - \alpha } \right)}}\int_{t_0}^t {\frac{{{f^{\left( n \right)}}\left( \tau  \right)}}{{{{\left( {t - \tau } \right)}^{\alpha  - n + 1}}}}{\rm{d}}\tau },
\end{eqnarray}

Alternatively, (\ref{Caputo}) can be rewritten in a form similar as the conventional Taylor series:
\begin{eqnarray}\label{con2}
{}_{{t_0}}^{\rm{C}}\mathscr{D}_t^\alpha f\left( t \right) = \sum\limits_{i = 0}^\infty  {\frac{{{f^{\left( {i + 1} \right)}}\left( {{t_0}} \right)}}{{\Gamma \left( {i + 2 - \alpha } \right)}}{{\left( {t - {t_0}} \right)}^{i + 1 - \alpha }}}.
\end{eqnarray}

The Riemann-Liouville's derivative \cite{podlubny1998fractional} with $\alpha $ for $f(t)$ is given by
\begin{eqnarray}\label{R-L}
{}^{\rm{RL}}_{t_0}{\mathscr D}_t^\alpha f\left( t \right) = \frac{{{{\rm{d}}^n}}}{{{\rm{d}}{t^n}}}\Big[ {\frac{1}{{\Gamma \left( {n - \alpha } \right)}}\int_{t_0}^t {\frac{{f\left( \tau  \right)}}{{{{\left( {t - \tau } \right)}^{\alpha  - n + 1}}}}{\rm{d}}\tau } } \Big].
\end{eqnarray}
or a series form as
\begin{eqnarray}\label{con5}
{}_{{t_0}}^{\rm{RL}}\mathscr{D}_t^\alpha f\left( t \right) = \sum\limits_{i = 0}^\infty  {\frac{{{f^{\left( {i} \right)}}\left( {{t_0}} \right)}}{{\Gamma \left( {i + 1- \alpha } \right)}}{{\left( {t - {t_0}} \right)}^{i  - \alpha }}}.
\end{eqnarray}

Suppose $f(t)$ to be a smooth convex function with a unique extreme point $t^*$. It is well known that each iterative step of the conventional GM \cite{boyd2004convex} is formulated as
 \begin{eqnarray}\label{GM}
{t_{k + 1}} = {t_k} - \rho {f^{\left( 1 \right)}}\left( {{t_k}} \right),
\end{eqnarray}
where $\rho>0$ is the iteration step size.

 The basic idea of FOGMs is then replacing the first order derivative in equation (\ref{GM}) by its fractional order counterpart, either using Caputo or Riemann Liouville's definition. However, it is shown that such a heuristic approach cannot guarantee the convergence capability of the algorithms \cite{pu2010fractional}. Thus we propose an alternative FOGM whose convergence can be guaranteed, which can be formulated as
 \begin{eqnarray}\label{fogm2}
{t_{k + 2}} = {t_{k+1}} - \rho {}_{{t_k}}^{C}\mathscr{D}_{{t_{k+1}}}^\alpha f\left( t \right),
\end{eqnarray}
where $0<\alpha<1$ and $\rho>0$.

By reserving the first item of ${}_{{t_k}}^{C}\mathscr{D}_{{t_{k+1}}}^\alpha f\left( t \right)$ in its infinite series form (\ref{con2}), the following FOGM is obtained
 \begin{eqnarray}\label{fogm3}
{t_{k + 2}} = {t_{k+1}} - \rho {f^{\left( 1 \right)}}\left( {{t_k}} \right){\left( {{t_{k + 1}} - {t_k}} \right)^{1 - \alpha }}.
\end{eqnarray}

Similar analysis can be applied for the Riemann Liouville's definition. Yet we have to reserve the second item of the infinite series form (\ref{con5}) since the first item contains the constant item of a function which should not influence the extreme point.

Assume that the algorithm is convergent, FOGM (\ref{fogm3}) can be further transformed into
 \begin{eqnarray}\label{fogm4}
{t_{k + 2}} = {t_{k+1}} - \rho {f^{\left( 1 \right)}}\left( {{t_{k+1}}} \right){\left( {{t_{k + 1}} - {t_k}} \right)^{1 - \alpha }}.
\end{eqnarray}

To make the variable step size $\rho{\left( {{t_{k + 1}} - {t_k}} \right)^{1 - \alpha }}>0$ always holds, FOGM (\ref{fogm4}) can be modified as
 \begin{eqnarray}\label{fogm21}
{t_{k + 2}} = {t_{k+1}} - \rho {f^{\left( 1 \right)}}\left( {{t_{k+1}}} \right){\left| {{t_{k + 1}} - {t_k}} \right|^{1 - \alpha }},
\end{eqnarray}
where $0<\alpha<2$ and $\rho>0$.

\begin{remark}
The proposed FOGM can be extended to the vector case directly. For a convex function $f(t),~t\in {\mathbb R}^n$, we can use the proposed FOGM to derive its extreme point and the algorithm is formulated as
 \begin{eqnarray}\label{fogm30}
{t_{k + 2}} = {t_{k+1}} - \rho \nabla {f\left( {{t_{k+1}}} \right){ |{t_{k + 1}} - {t_k}|}^{1 - \alpha }},
\end{eqnarray}
where $\nabla f(t_{k+1})$ denotes its gradient at $t_{k+1}$, $0<\alpha<2$, $\rho>0$, and $\theta ^ \alpha,~\theta \in {\mathbb R}^n$ denotes taking $\alpha$-th power law of each component.

\end{remark}

\section{Properties analysis of FOGM}\label{sec4}
The proposed FOGM (\ref{fogm21}) can guarantee a convergence to the extreme point, if it is convergent. Yet, there is no further properties analysis for FOGM (\ref{fogm21}). Thus in this section, we will discuss the properties of FOGM (\ref{fogm21}).

\begin{lemma}\label{lemma1}
For a strong convex function $f(t)$, $t_k$ must go across the extreme point for infinite times in FOGM (\ref{fogm21}) with $1<\alpha<2$.
\end{lemma}
\begin{proof}
Since $f(t)$ is strong convex, there exist a positive scalar $\lambda$ such that
\begin{eqnarray}\label{condition17}
\left| {{f^{\left( 1 \right)}}\left( x \right) - {f^{\left( 1 \right)}}\left( y \right)} \right| \ge \lambda \left| {x - y} \right|
\end{eqnarray}
for any $x$ and $y$ belonging to the definition domain of $f(t)$.

Then
\begin{eqnarray}\label{con31}
\begin{array}{rl}
\left| {{t_{k + 2}} - {t_{k + 1}}} \right| =&\hspace{-6pt} \left| {\rho {f^{\left( 1 \right)}}\left( {{t_{k + 1}}} \right)} \right|{\left| {{t_{k + 1}} - {t_k}} \right|^{1 - \alpha }}\\
 \ge&\hspace{-6pt} \rho\lambda \left| {{t_{k + 1}} - {t^*}} \right|{\left| {{t_{k + 1}} - {t_k}} \right|^{1 - \alpha }}.
\end{array}
\end{eqnarray}

We will then prove that $t_k$ must go across the extreme point by contradiction. Suppose $t_k$ does not go across $t^*$, then $t_k$ must get closer to the extreme point step by step since $f(t)$ is convex, which denotes that $\left| {{t_{k + 2}} - {t^*}} \right| < \left| {{t_{k + 1}} - {t^*}} \right|$ for any $k\ge3$.

Then following inequality must hold from (\ref{con31})
\begin{eqnarray}
\begin{array}{rl}
\left| {{t_{k + 2}} - {t_{k + 1}}} \right| =&\hspace{-6pt} \left| {{t_{k + 2}} - {t^*} - {t_{k + 1}} + {t^*}} \right|\\
 =&\hspace{-6pt} \left| {{t_{k + 1}} - {t^*}} \right| - \left| {{t_{k + 2}} - {t^*}} \right|\\
 \ge&\hspace{-6pt} \rho\lambda \left| {{t_{k + 1}} - {t^*}} \right|{\left| {{t_{k + 1}} - {t_k}} \right|^{1 - \alpha }},
\end{array}
\end{eqnarray}
which implies that
\begin{eqnarray}\label{con16}
\left| {{t_{k + 2}} - {t^*}} \right| \le \left(1- {\rho\lambda {{\left| {{t_{k + 1}} - {t_k}} \right|}^{1 - \alpha }} } \right)\left| {{t_{k + 1}} - {t^*}} \right|.
\end{eqnarray}
Yet, since it is assumed that $t_k$ never goes across $t^*$ and $f(t)$ is convex, $t_k$ must converge to the extreme point asymptotically. Thus for any arbitrary $\varepsilon>0$, there must exists an integer $N$ such that $|t_k-t^*|<\varepsilon$ holds for any $k>N$. Take $\varepsilon={\left( {\frac{{\rho \lambda }}{2}} \right)^{\frac{1}{{\alpha  - 1}}}}$, then $\rho\lambda|t_{k+1}-t_k|^{1-\alpha}>2$ and (\ref{con16}) does not hold. By contradiction, it is deduced that $t_k$ must go across $t^*$ from either side of $t^*$. Furthermore, this analysis will be repeated for infinite times.



\end{proof}
\begin{remark}
Lemma \ref{lemma1} reveals that FOGM (\ref{fogm21}) will go back and forth across the extreme point for a strong convex function. Yet, it may converge to the extreme point asymptotically from one side for a non-strong convex function since condition (\ref{condition17}) does not hold any more.
\end{remark}



\subsection{Convergence capability analysis of FOGM}

\begin{theorem}\label{theo2}
For a convex function $f(t)$ satisfying Lipschitz continuous gradient, FOGM (\ref{fogm21}) with $1<\alpha<2$ will always converge to a bounded region of $t^*$ for arbitrary $\rho$.
\end{theorem}
\begin{proof}
Since $f(t)$ satisfying Lipschitz continuous gradient, it is deduced that
\begin{eqnarray}
\left| {{f^{\left( 1 \right)}}\left( x \right) - {f^{\left( 1 \right)}}\left( y \right)} \right| \le \mu \left| {x - y} \right|,
\end{eqnarray}
for any $x$ and $y$ belonging to the definition domain of $f(t)$.

Define $\Delta_{k+1}=t_{k+1}-t_k$ and rewrite (\ref{fogm21}) as
\begin{eqnarray}
{|\Delta _{k + 2}|} =  | \rho {f^{\left( 1 \right)}}\left( {{t_{k + 1}}} \right)\Delta _{k + 1}^{1 - \alpha }|.
\end{eqnarray}

Then
\begin{eqnarray}\label{conyy}
\begin{array}{rl}
\left| {{\Delta _{k + 2}}} \right| =&\hspace{-6pt} \left| {\rho {f^{\left( 1 \right)}}\left( {{t_{k + 1}}} \right)\Delta _{k + 1}^{1 - \alpha }} \right|\\
 =&\hspace{-6pt} \left| {\rho \left[ {{f^{\left( 1 \right)}}\left( {{t_{k + 1}}} \right) - {f^{\left( 1 \right)}}\left( {{t^*}} \right)} \right]\Delta _{k + 1}^{1 - \alpha }} \right|\\
 \le &\hspace{-6pt}\left| {\rho \mu \left( {{t_{k + 1}} - {t^*}} \right)\Delta _{k + 1}^{1 - \alpha }} \right|.
\end{array}
\end{eqnarray}

Case 1: If $t_k$ goes across $t^*$ for only finite times, then there exists a sufficient large $N$ such that $t_k$ never goes across $t^*$ for $k>N$. Due to the convexity of $f(t)$ and the fact that $t_k$ never goes across $t^*$ for $k>N$, $\Delta_k$ must converge to $0$ and $t_k$ must be convergent. It is shown that the criteria for the convergence of (\ref{fogm21}) is $\mathop {\lim }\limits_{k \to \infty } \rho {f^{\left( 1 \right)}}\left( {{t_{k + 1}}} \right){\left| {{t_{k + 1}} - {t_k}} \right|^{1 - \alpha }} = 0$. Since $|t_{k+1}-t_k|^{1-\alpha}$ is nonzero with $1<\alpha<2$, it is concluded that $\mathop {\lim }\limits_{k \to \infty } {f^{\left( 1 \right)}}\left( {{t_{k + 1}}} \right) = 0$, which implies FOGM (\ref{fogm21}) will converge to $t^*$ asymptotically and the upper bound is zero.


Case 2: If $t_k$ goes across $t^*$ for infinite times, then one can find a sequence of $k_i,~i=1,2,\cdots $ such that $f^{(1)}(t_{k_i})f^{(1)}(t_{k_i+1})<0$. We will then prove that $|\Delta_{k_i+1}|$ is bounded.

Since $f^{(1)}(t_{k_i})f^{(1)}(t_{k_i+1})<0$ holds for each $k_i$, thus $|t_{k_i+1}-t^*|<|\Delta_{k_i+1}|$ and $|\Delta_{k_{i+1}}|\le|\Delta_{{k_i}+1}|$ hold. Thus for each $k_i$, (\ref{conyy}) can be transformed into
\begin{eqnarray} \label{con20}
\begin{array}{rl}
\left| {{\Delta _{{k_{i + 1}} + 1}}} \right| \le&\hspace{-6pt} \rho \mu \left| {{t_{{k_{i + 1}}}} - {t^*}} \right|{\left| {{\Delta _{{k_{i + 1}}}}} \right|^{1 - \alpha }}\\
 <&\hspace{-6pt} \rho \mu {\left| {{\Delta _{{k_{i }+1}}}} \right|^{2 - \alpha }}
\end{array}
\end{eqnarray}
where $t_{k_1+1}$ is the first time when $t_k$ goes across $t^*$ from one side.

Following equation can be obtained from (\ref{con20})
\begin{eqnarray}
\left| {{{\left( {\rho \mu } \right)}^{\frac{1}{{1-\alpha  }}}}{\Delta _{{k_{i + 1}} + 1}}} \right| < {\left| {{{\left( {\rho \mu } \right)}^{\frac{1}{{1-\alpha  }}}}{\Delta _{{k_{i }+1}}}} \right|^{2 - \alpha }}.
\end{eqnarray}
Take a transformation $z_i={\rm{ln}}\left| {{{\left( {\rho \mu } \right)}^{\frac{1}{{1-\alpha }}}}{\Delta _{{k_{i }} + 1}}} \right|$ and one can obtain that ${z_{i + 1}} < \left( {2 - \alpha } \right){z_i}$, which denotes that $\mathop {\lim }\limits_{i \to \infty } {z_i} \le 0$ since $0<2-\alpha<1$. Thus $\mathop {\lim }\limits_{i \to \infty } |{\Delta _{{k_{i + 1}}}}| \le {\left( {\rho \mu } \right)^{\frac{1}{{\alpha  - 1}}}}$.

Moreover, due to the convexity of function $f(t)$, $|t_j-t^*|<|\Delta_{k_i+1}|$ holds for any $k_i+1\le j \le k_{i+1}$. Thus $|\Delta_{k_i+1}|,~i=1,2,\cdots $ give an upper bound for $|t_k-t^*|$. Additionally, we have proven that $\mathop {\lim }\limits_{i \to \infty } {|\Delta _{{k_{i + 1}}}|} \le {\left( {\rho \mu } \right)^{\frac{1}{{\alpha  - 1}}}}$, which implies that $t_k$ will converge to a bounded region of $t^*$.

Combining Case 1 and 2, we complete the proof.

\end{proof}
\begin{remark}
Generally, a larger $\alpha>1$ will mean a worse convergence accuracy when $\rho\mu<1$, since ${\left| {\rho {\mu}} \right|^{\frac{1}{{\alpha  - 1}}}}$ is increasing with the increasing of $\alpha$. Yet, a larger $\alpha>1$ will give a better convergence capability when $\rho\mu>1$. Furthermore, ${\left| {\rho {\mu }} \right|^{\frac{1}{{\alpha  - 1}}}}$ tends to zero with $\alpha$ tending to $1$ when $\rho\mu<1$, which fits the conclusion of conventional GM well. Similarly, a larger step size $\rho$ gives a larger bound while a smaller step size $\rho$ gives a smaller bound.
\end{remark}
\begin{corollary}
Theorem \ref{theo2} still holds when $f(t)$ satisfies Lipschitz continuous gradient for $|t-t^*|>R$, $R>0$.
\end{corollary}
\begin{proof}
If $t_k$ only goes across $t^*$ for finite times, then $t_k$ will converge to $t^*$ asymptotically.

If $t_k$ goes across $t^*$ for infinite times but $|t_k-t^*|>R$ holds for finite times, then $|t_k-t^*|$ is bounded by $R>0$.

If $|t_k-t^*|>R$ for infinite times, then one can find a sequence $k_i,~i=1,2,\cdots$ such that $|t_{k_i+1}-t^*|>R$ holds for each $i$. We will then prove that $|\Delta _{t_{k_i+1}}|$ is bounded. Following inequality can be obtained for the $t_{k_i+1}>t^*$ case
\begin{eqnarray}
\begin{array}{rl}
\left| {{\Delta _{{k_{i + 1}} + 1}}} \right| =&\hspace{-6pt} \rho \left| {{f^{\left( 1 \right)}}\left( {{t_{{k_{i + 1}}}}} \right)} \right|{\left| {{\Delta _{{k_{i + 1}}}}} \right|^{1 - \alpha }}\\
 \le&\hspace{-6pt} \rho \left| {{f^{\left( 1 \right)}}\left( {{t_{{k_i} + 1}}} \right) - {f^{\left( 1 \right)}}\left( {R + {t^*}} \right)} \right|{\left| {{\Delta _{{k_{i + 1}}}}} \right|^{1 - \alpha }}\\
&\hspace{-6pt} + \rho \left| {{f^{\left( 1 \right)}}\left( {R + {t^*}} \right) - {f^{\left( 1 \right)}}\left( {{t^*}} \right)} \right|{\left| {{\Delta _{{k_{i + 1}}}}} \right|^{1 - \alpha }}\\
 \le&\hspace{-6pt} \rho \mu \left| {{t_{{k_{i + 1}}}} - \left( {R + {t^*}} \right)} \right|\\
 &\hspace{-6pt}+ \rho \left| {{f^{\left( 1 \right)}}\left( {R + {t^*}} \right)} \right|R{\left| {{\Delta _{{k_{i + 1}}}}} \right|^{1 - \alpha }}\\
 \le&\hspace{-6pt} \rho \theta \left| {{t_{{k_{i + 1}}}} - {t^*}} \right|{\left| {{\Delta _{{k_{i + 1}}}}} \right|^{1 - \alpha }}
\end{array}
\end{eqnarray}
where $\theta  = \max \{\mu,~|f^{(1)}({R+t^*})|\}$. Then similar to the Case 2 in the proof of Theorem \ref{theo2}, it is concluded that $|\Delta_{k_i+1}|$ is bounded. Similar analysis can be applied for the $t_{k_i+1}<t^*$ case.

From all the above analysis, it is concluded that either $|t_k-t^*|$ or $|\Delta_k|$ is bounded, which establishes the theorem.
\end{proof}
\begin{theorem}\label{corollary1}
For a strong convex function $f(t)$ which satisfies Lipschitz continuous gradient, $t_k$ cannot asymptotically converge to the extreme point but only converges to a bounded region of $t^*$.
\end{theorem}
\begin{proof}

Since $f(t)$ satisfies Lipschitz continuous gradient, FOGM (\ref{fogm21}) must converge to a bounded region of $t^*$ due to Theorem \ref{theo2}. We will prove that FOGM (\ref{fogm21}) cannot converge to $t^*$ asymptotically. Since $f(t)$ is strong convex, there exist a scalars $\lambda$ such that
\begin{eqnarray}
\left| {{f^{\left( 1 \right)}}\left( x \right) - {f^{\left( 1 \right)}}\left( y \right)} \right| \ge \lambda \left| {x - y} \right|
\end{eqnarray}
for any $x$ and $y$ belonging to the definition domain of $f(t)$. Similar to the proof of Theorem \ref{theo2}, one can obtain following inequality
\begin{eqnarray}\label{conzz}
\left| {{\Delta _{k + 2}}} \right| \ge \left| {\rho \lambda \left( {{t_{k + 1}} - {t^*}} \right)\Delta _{k + 1}^{1 - \alpha }} \right|.
\end{eqnarray}

Suppose $t_k$ converges to $t^*$ asymptotically. Thus for arbitrary $\varepsilon$, there exists an integer $N$ such that $|t_k-t^*|<\varepsilon$ for any $k>N$. Since $f(t)$ is strong convex, thus $t_k$ will go across $t^*$ for infinite times with $k>N$  from Lemma \ref{lemma1}. We will then prove the theorem by contradiction. If $f^{(1)}(t_k)f^{(1)}(t_{k+1})<0$ holds at some step $k$, then $|t_{k+1}-t^*|<|\Delta_{k+1}|$ and (\ref{conzz}) can be rewritten as
\begin{eqnarray}
\begin{array}{rl}
\left| {{\Delta _{k + 2}}} \right| \ge &\hspace{-6pt}\rho \lambda \left| {{t_{k+1}} - {t^*}} \right|{\left| {{\Delta _{k+1}}} \right|^{1 - \alpha }}\\
 > &\hspace{-6pt} \rho \lambda {\left| {{t_{k+1}} - {t^*}} \right|^{1 - \alpha }}|\Delta_{k+1}|.
\end{array}
\end{eqnarray}
Similarly, if $f^{(1)}(t_{k})f^{(1)}(t_{k+1})>0$ holds at some step $k>N$, then $|t_{k}-t^*|>|\Delta_{k+1}|$ and $|t_{k}-t^*|>|t_{k+1}-t^*|$ hold. Thus (\ref{conzz}) can be rewritten as
\begin{eqnarray}
\begin{array}{rl}
\left| {{\Delta _{k + 2}}} \right| \ge &\hspace{-6pt}\rho \lambda \left| {{t_{k+1}} - {t^*}} \right|{\left| {{\Delta _{k+1}}} \right|^{1 - \alpha }}\\
 > &\hspace{-6pt} \rho \lambda {\left| {{t_k} - {t^*}} \right|^{1 - \alpha }}|\Delta_{k+1}|.
\end{array}
\end{eqnarray}

Let $\varepsilon={\left( {\frac{{\rho \lambda }}{2}} \right)^{\frac{1}{{\alpha  - 1}}}}$, then $|\Delta_{k+1}|>2|\Delta_k|$ holds for any $k>N$. Thus $\Delta_k$ will finally be divergent, which contradicts to the assumption that $|\Delta_k|<2\varepsilon$. From above analysis, it is concluded that FOGM (\ref{fogm21}) will only converge to a region of $t^*$, which establishes the theorem.
\end{proof}

\begin{remark}
In the conventional GM, the convex function is supposed to be strong convex when talking about the convergence property. Yet, the property of strong convexity is to avoid the asymptotical convergence of FOGM (\ref{fogm21}) from the analysis of Corollary \ref{corollary1}. Thus for some non-strong convex function, FOGM (\ref{fogm21}) may still guarantee a asymptotical convergence, which will be discussed later.
\end{remark}

\subsection{Convergence rate analysis of FOGM}\label{subc}
In this subsection, we will discuss the convergence rate of FOGM (\ref{fogm21}) with different gradient order $\alpha$ qualitatively.

\begin{itemize}
    \item[1)] With $0<\alpha<1$, if $|\Delta_k|>1$, the convergence rate will be faster than the conventional case since $|\Delta_k|^{1-\alpha}>1$ and the step size $\rho |\Delta_k|^{1-\alpha}$ is larger than $\rho$. Yet, if $|\Delta_k|<1$, the convergence rate will be rather slower since the step size $\rho |\Delta_k|^{1-\alpha}$ is smaller than $\rho$. Particularly, if FOGM (\ref{fogm21}) with $0<\alpha<1$ is convergent, the convergence rate is very slow when $t_k$ is close to $t^*$ since $\Delta_k$ is very small.


    \item[2)] With $1<\alpha<2$, if $|\Delta_k|>1$, the convergence rate will be slower than the conventional case since $|\Delta_k|^{1-\alpha}<1$ and the step size $\rho |\Delta_k|^{1-\alpha}$ is smaller than $\rho$. Yet, if $|\Delta_k|<1$, the convergence rate will be rather faster since the step size $\rho |\Delta_k|^{1-\alpha}$ is much larger than $\rho$. Moreover, FOGM (\ref{fogm21}) shows a great convergence property but with a lower convergence accuracy.
    \item[3)] The conventional GM with $\alpha=1$ can be viewed as a trade-off in the convergence rate between $|\Delta_k|>1$ and $|\Delta_k|\le 1$.
\end{itemize}
\section{Modified FOGM}\label{sec8}
Though algorithm (\ref{fogm21}) can guarantee a great convergence property with $\alpha>1$, it can only converge to a small neighbourhood of $t^*$, which is undesired. Yet, the convergence accuracy would be improved a lot with algorithm {\rm{(\ref{fogm21})}} modified and the novel FOGM can be formulated as
 \begin{eqnarray}\label{fogm22}
{t_{k + 2}} = {t_{k+1}} - \rho {f^{\left( 1 \right)}}\left( {{t_{k+1}}} \right){\left(\left| {{t_{k + 1}} - {t_k}} \right|+\delta \right)^{1 - \alpha }},
\end{eqnarray}
where $1<\alpha<2$ and $\delta$ is a positive scalar.
\begin{theorem}\label{theo41}
For a convex function satisfying Lipschitz continuous gradient, modified FOGM (\ref{fogm22}) will converge to a bounded region of $t^*$ for arbitrary step size $\rho$.
\end{theorem}
\begin{proof}
Similar to the proof of Theorem \ref{theo2}, if $t_k$ goes across $t^*$ for only finite times, it will converge to the extreme point asymptotically, whose bound is zero. If $t_k$ goes across $t^*$ for infinite times, then similar to (\ref{con20}), one can find a sequence of $k_i,~i=1,2,\cdots$ such that
 \begin{eqnarray} \label{con40}
\begin{array}{rl}
\left| {{\Delta _{{k_{i + 1}} + 1}}} \right| < &\hspace{-6pt}\rho \mu \left| {{\Delta _{{k_i} + 1}}} \right|{\left( {\left| {{\Delta _{{k_i} + 1}}} \right| + \delta } \right)^{1 - \alpha }}\\
 <&\hspace{-6pt} \rho \mu {\left| {{\Delta _{{k_i} + 1}}} \right|^{2 - \alpha }}
\end{array}
\end{eqnarray}
Thus similar to the analysis in Theorem \ref{theo2}, it is concluded that $|\Delta_{k_i+1}|$ will be bounded, which denotes that $t_k$ will converge to a bounded region of $t^*$. This completes the proof.
\end{proof}

From condition (\ref{con40}), it is concluded that the bound of the converge region will be smaller with $\delta$ added. And if $\delta$ is sufficient large, then $t_k$ will converge to the extreme point asymptotically all the time and following theorem holds.

\begin{theorem}\label{theo3}
For a convex function $f(t)$ satisfying Lipschitz continuous gradient, algorithm {\rm{(\ref{fogm22})}} will converge to the extreme point asymptotically with $\delta$ satisfying
\begin{eqnarray}\label{cond1}
\left| {\rho \mu \delta^{1 - \alpha }} \right|<1.
\end{eqnarray}
\end{theorem}
\begin{proof}
If $t_k$ goes across $t^*$ for only finite times, it will converge to $t^*$ asymptotically due to the convexity of $f(t)$.

If $t_k$ goes across $t^*$ for infinite times, then one can find a sequence of $k_i,~i=1,2,\cdots$ such that
 \begin{eqnarray} \label{con41}
\begin{array}{rl}
\left| {{\Delta _{{k_{i + 1}} + 1}}} \right| < &\hspace{-6pt}\rho \mu \left| {{\Delta _{{k_i} + 1}}} \right|{\left( {\left| {{\Delta _{{k_i} + 1}}} \right| + \delta } \right)^{1 - \alpha }}\\
 <&\hspace{-6pt} \rho \mu \left| {{\Delta _{{k_i} + 1}}} \right|{\delta}^{1-\alpha}.
\end{array}
\end{eqnarray}
If $\left| {\rho \mu \delta^{1 - \alpha }} \right|<1$ holds, then $\left| {{\Delta _{{k_{i + 1}} + 1}}} \right| < \left| {{\Delta _{{k_i} + 1}}} \right|$ holds all the time, which denotes that $|\Delta_{k_i}|$ will converge to zero asymptotically. Thus $t_k$ must converge to $t^*$ asymptotically.

All the above analysis well implies the theorem.

%
\end{proof}
\begin{remark}
If $\delta$ is too small, it may not guarantee the asymptotical convergence. Yet, if $\delta$ is too large, the convergence rate may be much slower. In fact, if $\delta>1$, then the step size $\rho \left(|\Delta_i|+\delta\right)^{1-\alpha}<\rho$ always holds with $1<\alpha<2$, which denotes that the convergence rate is slower than the conventional case.
\end{remark}
\begin{remark}\label{remark8}
$\delta  > {\left| {\rho {\mu }} \right|^{\frac{1}{{\alpha  - 1}}}}$ can guarantee the asymptotical convergence of FOGM (\ref{fogm22}), where ${\left| {\rho {\mu }} \right|^{\frac{1}{{\alpha  - 1}}}}$ is the upper bound of the convergent region shown in Theorem \ref{theo2}. Thus one can find that $\delta$ guarantees the asymptotical convergence after $\Delta_k$ goes into the bounded region.
\end{remark}

Furthermore, $|\Delta_k|$ will soon become smaller than $1$ since step size $\rho$ is usually set sufficiently small to guarantee the convergence property. Thus FOGM (\ref{fogm22}) with $1<\alpha<2$ can usually present a faster convergence rate when $t_k$ is close to the extreme point. Yet, if some extreme conditions such as step size $\rho$ is large and initial iterative point is far away from the extreme point are considered, $|\Delta_k|$ will be larger than $1$ at the beginning and FOGM (\ref{fogm22}) with $1<\alpha<2$ may converge slower than the $0<\alpha\le1$ case. Considering the potential faster convergence rate at the beginning for the $0<\alpha<1$ case, following switching FOGM can be obtained
 \begin{eqnarray}\label{fogm23}
{t_{k + 2}} = {t_{k+1}} - \rho {f^{\left( 1 \right)}}\left( {{t_{k+1}}} \right){\left(\left| {{t_{k + 1}} - {t_k}} \right|+\delta \right)^{1 - \alpha }},
\end{eqnarray}
where $\alpha$ and $\delta$ are set as $\alpha  > 1,\delta  \ge {\left| {\rho \mu } \right|^{\frac{1}{{\alpha  - 1}}}}$ thereafter once $\left| {{t_{k }} - {t_{k-1}}} \right| < 1$ or ${f^{(1)}}({t_2}){f^{(1)}}({t_{k +}}) < 0$ holds at some step $k+1$ and set as $\alpha<1$ and $\delta=0$ for the other cases.
\begin{theorem}\label{theorem5}
For a convex function satisfying Lipschitz continuous gradient, modified FOGM (\ref{fogm23}) will guarantee a global asymptotical convergence all the time.
\end{theorem}
\begin{proof}
If $|\Delta_k|<1$ or $f^{(1)}(t_2)f^{(1)}(t_{k})<0$ holds at step $k$, then FOGM (\ref{fogm23}) is switched to $1<\alpha<2$ case and $\delta$ can guarantee a asymptotical convergence with $1<\alpha<2$ as shown in Theorem \ref{theo3}.

Moreover, either of the conditions $|\Delta_k|<1$ and $f^{(1)}(t_2)f^{(1)}(t_{k})<0$ must happen. If $|\Delta_k|<1$ never happens, then $t_k$ must go across the extreme point from either side of $t^*$ due to the convexity of $f(t)$, which denotes that $f^{(1)}(t_2)f^{(1)}(t_{k})<0$ must hold at some step $k$. Thus, FOGM (\ref{fogm23}) must converge to the extreme point asymptotically with arbitrary step size $\rho$.
\end{proof}
\begin{remark}
Switching FOGM (\ref{fogm23}) shows a faster convergence rate than conventional FOGM with $0<\alpha<2$. Though any step size $\rho$ can be designed, it is better to design a suitable step size $\rho$ with which $t_k$ will not go across the extreme point significantly. If $\rho$ is too large, then $f^{(1)}(t_2)f^{(1)}(t_{k})<0$ will hold at step $k$ where $t_k$ is far away from $t^*$. But FOGM (\ref{fogm23}) has already been switched to $1<\alpha<2$ case, which may result in a slower convergence rate as discussed in Subsection \ref{subc}.
\end{remark}
\begin{remark}
We have to address here that the condition $f^{(1)}(t_2)f^{(1)}(t_{k})<0$ is to avoid the divergence of FOGM (\ref{fogm23}). In fact, such condition can be omitted to obtain an even faster convergence rate. But the convergence property of FOGM (\ref{fogm23}) may not be guaranteed.
\end{remark}

\section{Some extensive discussion}\label{sec9}
Generally, many convex functions do not satisfy Lipschitz continuous gradient or are not strong convex. Thus in this section, we will extend such conventional concepts to a more general case.

\begin{theorem}\label{theo51}
For a convex function satisfying $p$-order Lipschitz continuous gradient, FOGM (\ref{fogm21}) with $p<\alpha<1+p$ will always converge to a bounded region of $t^*$ for arbitrary $\rho$. Moreover, the upper bound is ${\left| {\rho {\mu }} \right|^{\frac{1}{{\alpha  - p}}}}$.
\end{theorem}
The proof of Theorem \ref{theo51} can be obtained in the same way as Theorem \ref{theo2}. In fact, it may be tough or even impossible for a function to satisfy $p$-th order Lipschitz continuous gradient globally. Yet, if the condition holds for arbitrary $x,~y\in\{z||z-t^*|>R\}$ where $R$ is a positive scalar, then FOGM (\ref{fogm21}) can still converge to a bounded region for arbitrary $\rho$.


\begin{theorem}\label{theo52}
For a $p$-order strong convex function, a necessary condition for the asymptotical convergence of FOGM (\ref{fogm21}) is $\alpha\le p$.
\end{theorem}
The proof of Theorem \ref{theo52} can be obtained in the same way as Theorem \ref{corollary1}. Generally, it is tough for a function to satisfy $p$-th order strong convex globally. In fact, for a function which is $p$-order strong convex around $t^*$, Theorem \ref{theo52} still holds.

The conditions of local $p$-order strong convex and $p$-order Lipschitz continuous gradient are generally tough to determine. Yet, Theorem \ref{theo51} and \ref{theo52} do give a general form suitable for more convex functions and deepen our insight of GM.
\begin{corollary}\label{coro2}
For a convex function which is $p$-order Lipschitz continuous gradient where $p<1$ for $|t-t^*|>R$, $R>0$, then the conventional GM will never go to infinity for arbitrary step size $\rho>0$.
\end{corollary}
\begin{corollary}\label{coro3}
For a convex function which is $p$-order strong convex around the extreme point where $p>1$, then the conventional GM cannot guarantee a asymptotical convergence to the extreme point but converge to a bounded region about the extreme point.
\end{corollary}

\begin{remark}
Corollary \ref{coro2} and \ref{coro3} demonstrate that the conventional GM may still exist some questions when handling some specific convex functions, such as $f(t)={|t|}^{\frac{4}{3}}$. Yet, to the largest knowledge of the authors, the questions have not been reported before. Thus, introducing FOGM not only can improve the convergence performances of GM, but also is the natural extension of conventional GM. And it does provide detailed analysis when the conventional GM is used for a non-strong convex function.
\end{remark}

\section{Illustrative examples}\label{sec5}
In this section, we will present some typical examples to demonstrate the conclusions of proposed theorems.

\begin{example}\label{ex1}
Consider the simplest strong convex function $f(t)=(t-c)^2$ which satisfying Lipschitz continuous gradient and $\mu=\lambda=2$. Take $\rho=0.01,~c=3,~t_1=-1$, and $t_2=0$ when simulating.
\begin{figure}
\centering
\includegraphics[width=0.5\textwidth]{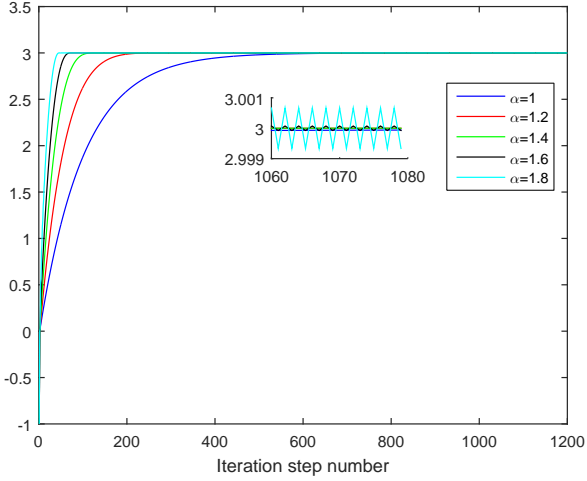}
\caption{Iteration results in Example \ref{ex1}}\label{f1}
\end{figure}

\begin{table}[hb]
\begin{center}
\caption{Some typical points with different $\alpha$ in Example \ref{ex1}}\label{tab1}
\begin{tabular}{c|cccccc}
\hline
\backslashbox{$t_k-t^*$}{Step $k$} & 1060 & 1061 & 1062 & 1063 & 1064 &1065\\\hline
$\alpha=1$ $(\times 10^{-5})$ & 7.23 &7.16 & 7.09 & 6.95& 6.88 &6.81\\\hline
$\alpha=1.2$ $(\times 10^{-12})$ & -1.56 &1.56 &-1.56 & 1.56& -1.56 &1.56\\\hline
$\alpha=1.4$ $(\times 10^{-7})$& -8.84 &8.84 &-8.84 & 8.84& -8.84 &8.84\\\hline
$\alpha=1.6$ $(\times 10^{-5})$& -7.31 &7.31 &-7.31 & 7.31& -7.31 &7.31\\\hline
$\alpha=1.8$ $(\times 10^{-4})$& -6.65 &6.65 &-6.65 &6.65& -6.65 &6.65\\\hline
\end{tabular}
\end{center}
\end{table}

Results are shown in Fig. \ref{f1} and TABLE \ref{tab1}. Following conclusions can be derived:
\begin{itemize}
    \item[1)] A larger $\alpha$ gives a faster convergence rate from Fig. \ref{f1}.
    \item[2)] With $\alpha>1$, FOGM (\ref{fogm21}) cannot converge to the extreme point asymptotically but a small neighborhood of the extreme point. Moreover, the larger $\alpha$ always means a worse convergence accuracy as shown in TABLE \ref{tab1}.
    \item[3)] Calculate the value of ${\left| {\rho \mu} \right|^{\frac{1}{{\alpha  - 1}}}}$ for $\alpha=1.2,~\alpha=1.4,~\alpha=1.6$ and $\alpha=1.8$ and the results are $3.2\times 10^{-9},~5.7\times 10^{-5},~1.5\times 10^{-3}$, and $7.5\times 10^{-3}$, respectively. Though the estimated bounds are larger than the real ones, it does give some information about convergence accuracy in advance.
\end{itemize}
\end{example}
\begin{example}\label{ex2}
Consider the same function in Example \ref{ex1}, take $\alpha=1.5$, $c=3,~t_1=-1$, and $t_2=0$ when simulating.

\begin{figure}
\centering
\includegraphics[width=0.5\textwidth]{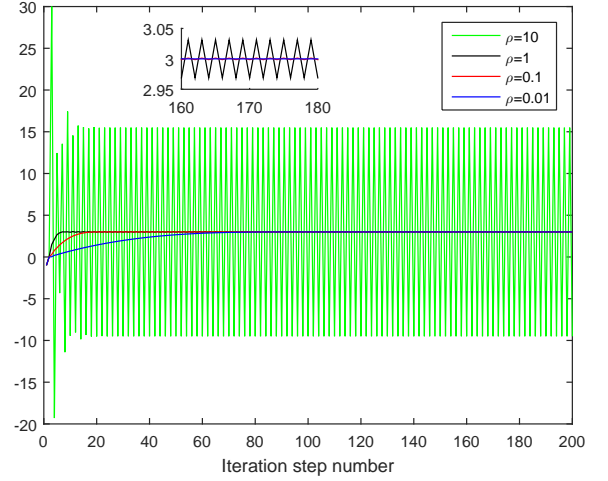}
\caption{Iteration results in Example \ref{ex2}}\label{f2}
\end{figure}
\end{example}

Results are shown in Fig. \ref{f2}. When $\rho=10$, the conventional GM has already gone divergent, which is not shown here. Yet, FOGM (\ref{fogm21}) never goes divergent but converges to a neighbourhood of the extreme point. Moreover, a larger $\rho$ means a worse convergence accuracy. No matter how large is the bound of convergence accuracy, it will never go to infinity, which well demonstrates the conclusion of Theorem \ref{theo2}.

\begin{example}\label{ex3}
Consider the same function in Example \ref{ex1} and the modified FOGM (\ref{fogm22}). Take $\alpha=1.5$, $\rho=0.1$, $c=3,~t_1=-1$, and $t_2=0$. From the analysis in Remark \ref{remark8}, $\delta$ can be set as ${\left| {\rho {\mu }} \right|^{\frac{1}{{\alpha  - 1}}}}$ with $\mu=2$ and $\alpha=1.5$. Thus take $\delta=0.04$ when simulating.
\begin{figure}
\centering
\includegraphics[width=0.5\textwidth]{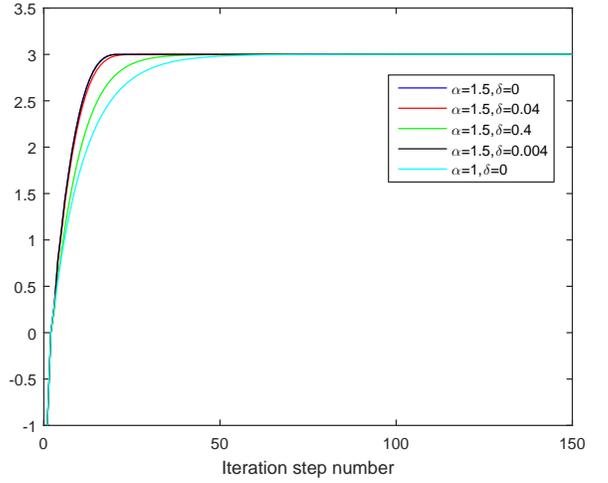}
\caption{Iteration results in Example \ref{ex3}}\label{f3}
\end{figure}
\begin{table}[hb]
\begin{center}
\caption{Some typical points with different $\delta$ in Example \ref{ex3}}\label{tab2}
\begin{tabular}{c|cccccc}
\hline
\backslashbox{$t_k-t^*$}{Step $k$} & 140 & 141 & 142 & 143 & 144 &145\\\hline
$\delta=0$ $(\times 10^{-3})$ & 1.3 &-1.3 & 1.3 & -1.3& 1.3 &-1.3\\\hline
$\delta=0.04$ & 0 &0 &0 & 0& 0 &0\\\hline
$\delta=0.4$ $(\times 10^{-10})$& 2.67 &2.25 &1.89 & 1.60& 1.34 &1.13\\\hline
$\delta=0.004$ $(\times 10^{-16})$& -4.44 &4.44&-4.44 &4.44& -4.44 &4.44\\\hline
\end{tabular}
\end{center}
\end{table}

Results are shown in Fig. \ref{f3} and TABLE \ref{tab2}. Following conclusions can be directly derived:
\begin{itemize}
    \item[1)] Smaller $\delta$ means a faster convergence rate. Moreover, if $\delta<1$, the convergence rate is always faster than the conventional GM, i.e., $\alpha=1$ as shown in Fig. \ref{f3}.
    \item[2)] If $\delta$ is selected too small like $\delta=0.004$ case, it cannot guarantee the asymptotical convergence but only improves the convergence accuracy. If $\delta$ is selected too large like $\delta=0.4$ case, it can guarantee the asymptotical convergence but the convergence rate is much slower.
    \item[3)] Our estimation for $\delta={\left| {\rho {\mu }} \right|^{\frac{1}{{\alpha  - 1}}}}$ can guarantee the asymptotical convergence with a satisfying convergence rate from Fig. \ref{f3} and TABLE \ref{tab2}, which validates the effectiveness of modified FOGM (\ref{fogm22}).
\end{itemize}
\end{example}

\begin{example}\label{ex4}
In this example, we will show the convergence rate in some extreme conditions. Consider the same function in Example \ref{ex1}. Here, different extreme points and different FOGMs are considered. If $0<\alpha<1$, then FOGM (\ref{fogm22}) with $\delta=0$ is used. If $1<\alpha<2$, then FOGM (\ref{fogm22}) with $\delta={\left| {\rho {\mu }} \right|^{\frac{1}{{\alpha  - 1}}}}$ is used. When it is mentioned switching FOGM, FOGM (\ref{fogm23}) is considered. Take $~t_1=0,~t_2=1,~\rho=0.01$ and gradient orders of the switching FOGM are $\alpha_1=0.7$ and $\alpha_2=1.3$ when simulating.

\begin{figure}
\centering
\includegraphics[width=0.5\textwidth]{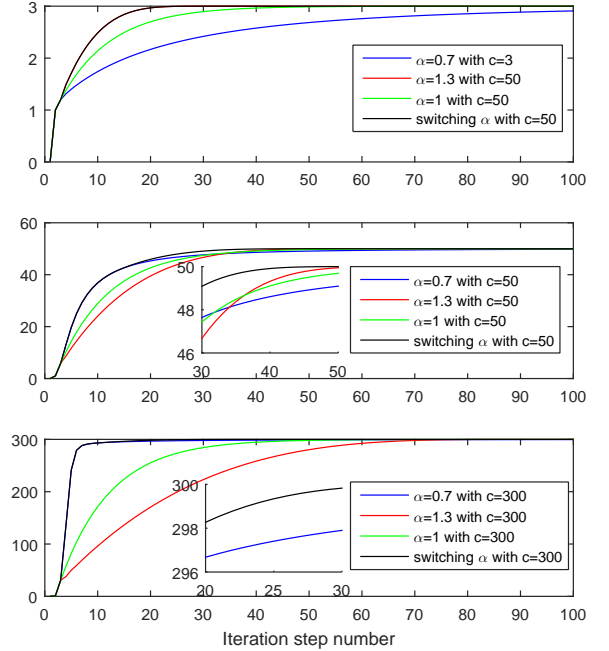}
\caption{Iteration results in Example \ref{ex4}}\label{f4}
\end{figure}

Results are shown in Fig. \ref{f4}. Following conclusions can be directly obtained:
\begin{itemize}
    \item[1)] FOGM (\ref{fogm22}) with $0<\alpha<1$ may show a faster convergence rate at the beginning in some situations resulting in a large $\Delta_k$, like a large step size $\rho$ or a large $f^{(1)}(t_k)$ as shown in the third sub-figure of Fig. \ref{f4}.
    \item[2)] Though FOGM (\ref{fogm22}) with $0<\alpha<1$ may show a faster convergence rate at the beginning, its convergence rate when $t_k$ is close to $t^*$ is rather worse, which is shown in the first sub-figure of Fig. \ref{f4}.
    \item[3)] Switching FOGM (\ref{fogm23}) shows a satisfying convergence rate both at the beginning or when $t_k$ is close to $t^*$. When $|\Delta_k|<1$ holds all the time, FOGM (\ref{fogm23}) shows the same convergence rate as FOGM (\ref{fogm22}) with $1<\alpha<2$ as shown in the first sub-figure of Fig. \ref{f4}.
\end{itemize}
\end{example}
\begin{example}\label{ex5}
In this example, we will compare the convergence property of different FOGMs. Consider the same function in Example \ref{ex1}. Different step sizes $\rho$ are considered. Take $t_1=0,~t_2=1$, $c=100$, and switching gradient orders are $\alpha_1=0.7$ and $\alpha_2=1.3$ when simulating.

\begin{figure}
\centering
\includegraphics[width=0.5\textwidth]{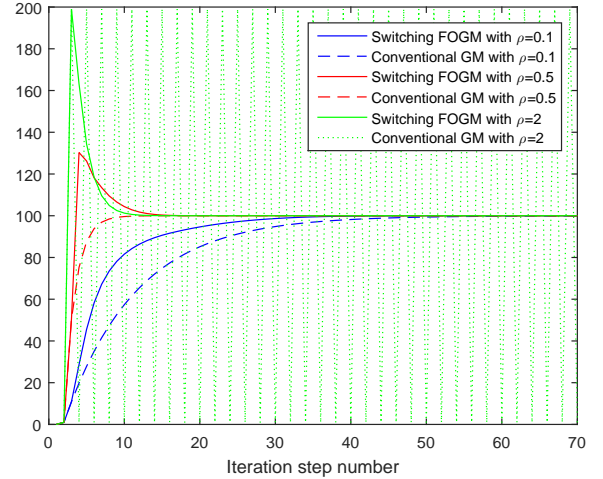}
\caption{Iteration results in Example \ref{ex5}}\label{f5}
\end{figure}

Results are shown in Fig. \ref{f5}. Following conclusions can be derived:
\begin{itemize}
    \item[1)] For FOGM (\ref{fogm23}), a larger $\rho$ does not mean a faster convergence rate. Thus a suitable $\rho$ will be the one with which $t_k$ does not go across $t^*$ significantly.
    \item[2)] FOGM (\ref{fogm23}) shows a great convergence property for arbitrary step size $\rho$, which demonstrates the conclusion of Theorem \ref{theorem5}. Yet, the conventional GM will be divergent if $\rho>2$ as shown in Fig. \ref{f5}.
    \item[3)] Here, we do not show the convergence property of FOGM (\ref{fogm21}) with $0<\alpha<1$. It is even worse than the conventional case since one can validate that FOGM (\ref{fogm21}) with $0<\alpha<1$ is divergent with $\rho>1.1$.
\end{itemize}
\end{example}

\begin{example}\label{ex6}
Consider a special convex function $f(t)=|t-c|^{\frac{4}{3}}$, which is not strong convex or satisfying Lipschitz continuous gradient globally. Yet, one can validate that $|f^{(1)}(x)-f^{(1)}(y)|\le |x-y|^{0.4}$ holds for $x,~y\in\{|t|>1000\}$. Take $t_1=-1,~t_2=0$, and $c=100$ when simulating.
\begin{figure}
\centering
\includegraphics[width=0.5\textwidth]{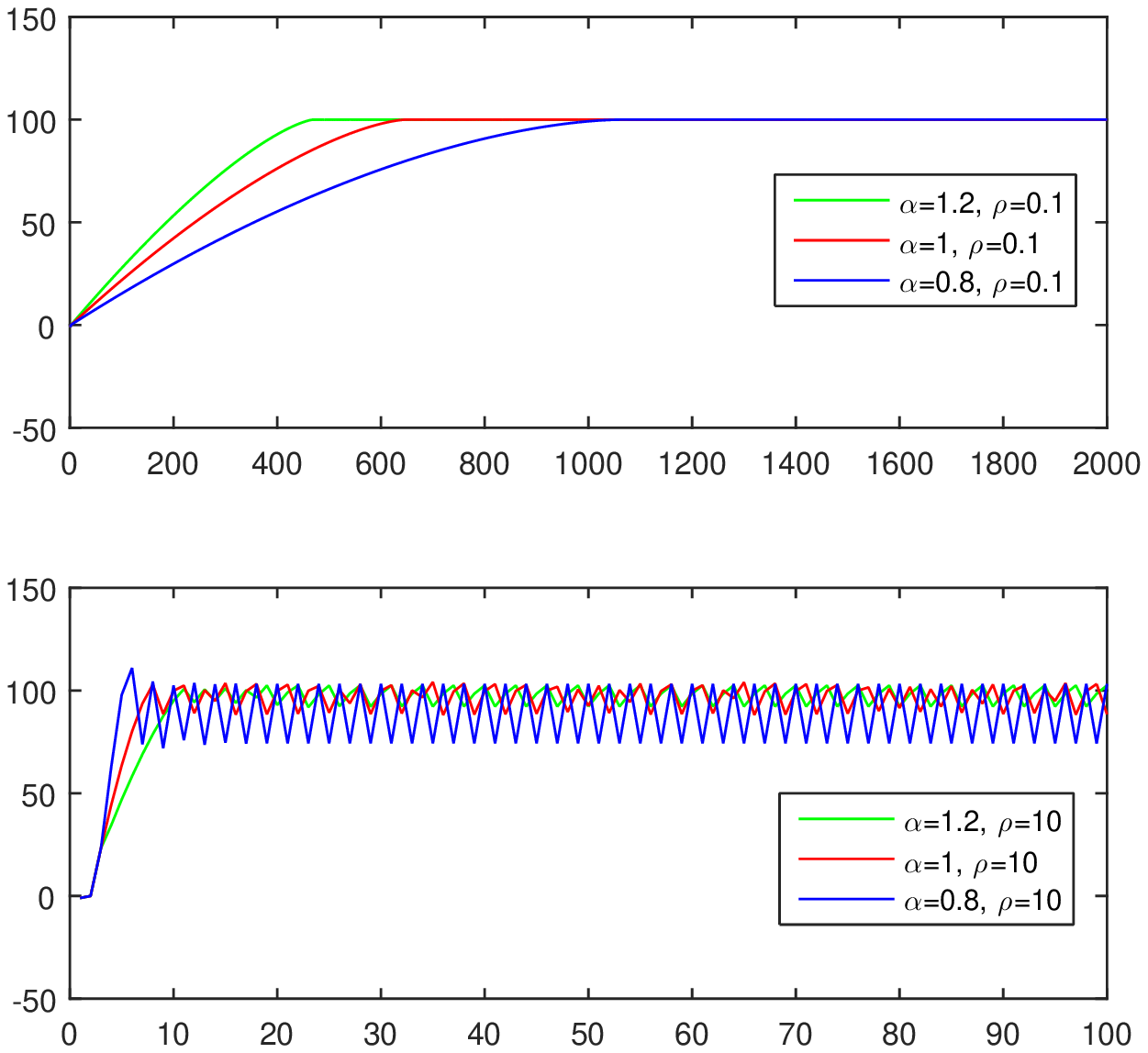}
\caption{Iteration results in Example \ref{ex6}}\label{f6}
\end{figure}

Results are shown in Fig. \ref{f6}. It is concluded that FOGM (\ref{fogm21}) with $0.4<\alpha<1.4$ will never go divergent but converge to a bounded region around $t^*$, which validates the conclusion of Theorem \ref{theo51}. Moreover, the conventional GM with $\alpha=1$ still cannot guarantee an asymptotical convergence, which implies that the order of strong convexity for $f(t)=|t-c|^{\frac{4}{3}}$ is less than $1$ or more accurately $0.8$.

Though the accurate order of local strong convexity around $t^*$ is tough to determine, we can find the approximate order by simulation. Take $t_1=-1,~t_2=0$,and $\rho=2$ when simulating. Results are shown in Fig. \ref{f62} and TABLE \ref{tab3}. With $\alpha=0.332$, $t_k$ can converge to the extreme point asymptotically. Yet, with $\alpha=0.334$, $t_k$ only converges to a bounded region of the extreme point. Thus the approximate order of strong convexity is $\frac{1}{3}$.

\begin{figure}
\centering
\includegraphics[width=0.5\textwidth]{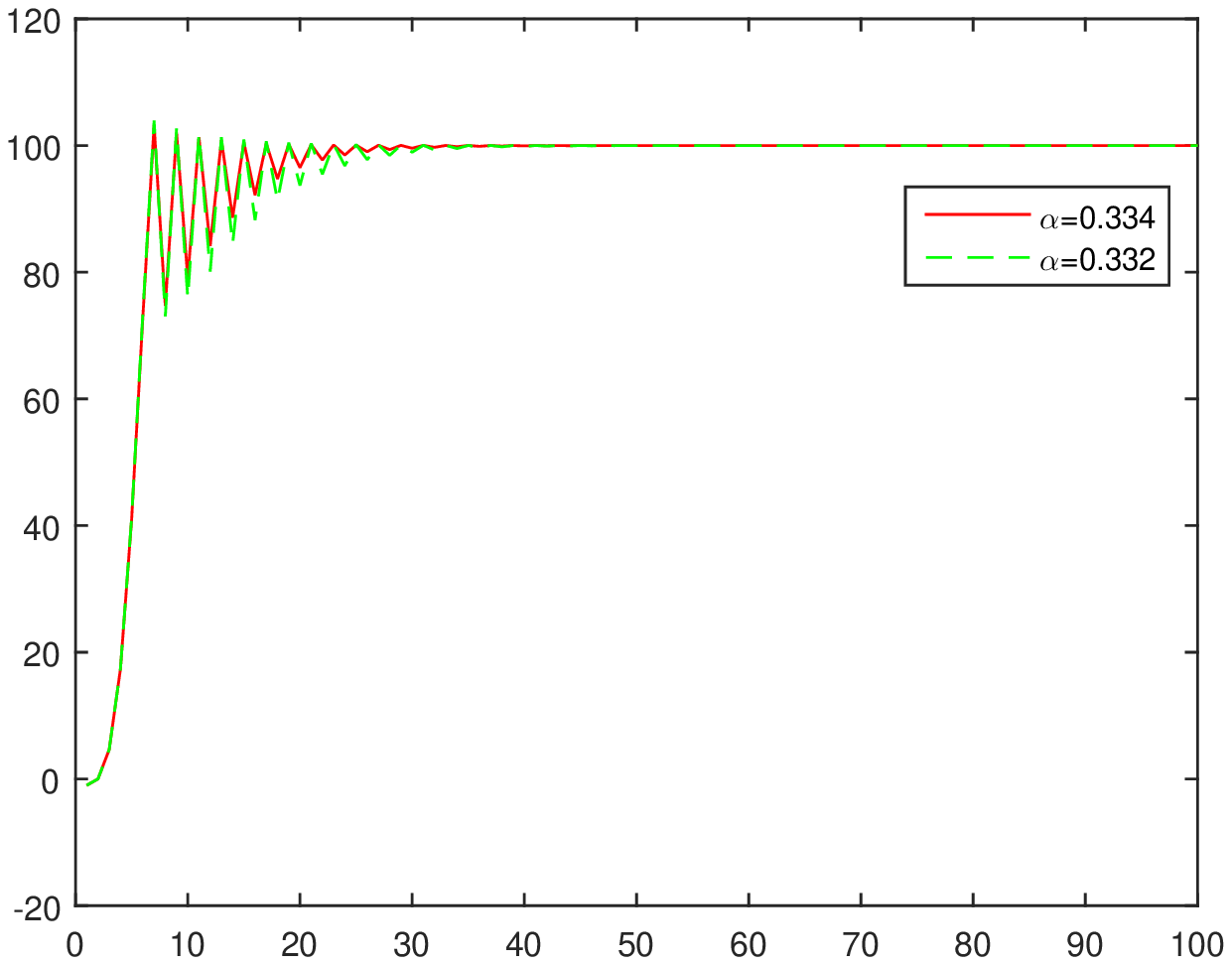}
\caption{Iteration results in Example \ref{ex6}}\label{f62}
\end{figure}

\begin{table}[!htb]
\begin{center}
\caption{Some typical points with different $\delta$ in Example \ref{ex6}}\label{tab3}
\begin{tabular}{c|cccccc}
\hline
\backslashbox{$t_k-t^*$}{Step $k$} &455& 456 & 457 & 458 & 459 & 460 \\\hline

$\alpha=0.334$ $(\times 10^{-14})$& 2.84 &-22.7 &2.84 & -22.7& 2.84 &-22.7\\\hline

\backslashbox{$t_k-t^*$}{Step $k$} &155& 156 & 157 & 158 & 159 & 160 \\\hline
$\alpha=0.332$ $(\times 10^{-13})$ &-2.56 & 14.2 &-1.84 &0 & 0& 0 \\\hline
\end{tabular}
\end{center}
\end{table}

\end{example}
\section{Conclusion}\label{sec6}
In this paper, we carefully analyze the convergence capability, convergence accuracy, and convergence rate of a novel FOGM. Due to the special properties of FOGM with $0<\alpha<1$ and $1<\alpha<2$, a switching FOGM is proposed, which shows superiorities in both convergence rate and convergence capability. Moreover, we extend the conventional concepts of Lipschitz continuous gradient and strong convex to a more general case and all the proposed conclusion are extended to a more general case. Finally, numerous simulation examples demonstrates the effectiveness of proposed methods fully. A promising future topic can be directed to apply the proposed FOGM in some related fields like LMS filter and system identification.

\bibliographystyle{IEEEtran}
\bibliography{IEEEabrv,FOEP}







\end{document}